\documentclass[11pt, reqno]{amsart}
\usepackage[applemac]{inputenc}
\usepackage[cm]{aeguill}
\usepackage{amsfonts}
\usepackage{latexsym}
\usepackage[dvips]{graphicx}
\usepackage{mathrsfs}

\usepackage[active]{srcltx}

\usepackage{amsmath}
\usepackage{amssymb}
\usepackage{amsthm}
\usepackage{nicefrac}
\usepackage[T1]{fontenc}

\newtheorem{df}{Definition}[section]
\newtheorem{lm}[df]{Lemma}
\newtheorem{pr}[df]{Proposition}
\newtheorem{Th}[df]{Theorem}
\newtheorem{co}[df]{Corollary}
\newtheorem{rem}[df]{Remark}

\newenvironment{oss}{\begin{rem} \begin{rm}}{\end{rm} \end{rem}}

\newcommand{\cchi}{\mbox{\large $\chi$}}
\newcommand{\cyl}{(0,+\infty)\times\R^n}
\newcommand{\ccyl}{[0,+\infty)\times\R^n}
\newcommand{\Tcyl}{(0,T)\times\R^n}

\newcommand{\Tlccyl}{[0,T)\times\R^n}
\newcommand{\e}{\varepsilon}
\newcommand{\eps}{\varepsilon}

\newcommand{\R}{\mathbb{R}}
\newcommand{\N}{\mathbb{N}}
\newcommand{\T}{\mathbb T}

\newcommand{\M}{\mathcal{M}}

\newcommand{\HH}{\mathcal{H}}

\newcommand{\D}[1]{\mbox{\rm #1}}

\newcommand{\Sub}{\mathscr{S}}

\renewcommand{\bar}{\overline}

\begin{document}

\title{On the (non) existence of viscosity solutions\\ of multi--time Hamilton--Jacobi equations}

\author{A. DAVINI, M. ZAVIDOVIQUE
%
}
\address{Dip. di Matematica, {Sapienza} Universit\`a di Roma,
P.le Aldo Moro 2, 00185 Roma, Italy}
\email{davini@mat.uniroma1.it}
\address{
IMJ-PRG (projet Analyse Alg\' ebrique), UPMC,  
4, place Jussieu, Case 247, 75252 Paris C\' edex 5, France}
\email{zavidovique@math.jussieu.fr}

\thanks{The second author is financed by ANR-12-BLAN-WKBHJ}  

\begin{abstract}
We prove that the multi--time Hamilton--Jacobi equation in general cannot be solved in the viscosity sense, in the non-convex setting, even when the Hamiltonians are in involution.
\end{abstract}

\maketitle

\section*{Introduction}
There are two important Hamilton--Jacobi equations associated with a continuous function $H$ defined on the cotangent bundle of a smooth Riemannian manifold $M$: 
the stationary equation, that is a nonlinear PDE of the kind
$$
H(x,d_x u)=c\qquad\hbox{in $M$},
$$
where $c$ is a real constant and the unknown $u$ is a real function defined on $M$;
and the evolutionary equation, which is a time-dependent equation of the form 
\begin{eqnarray}\label{intro evolutive eq}
\left\{ 
\begin{array}{ll}
\partial_t u +H(x,d_x u) = 0\quad & \hbox{in $(0,+\infty)\times M$}\bigskip \\
u(0,x)=u_0(x), & \hbox{in $M$},
\end{array}
\right.
\end{eqnarray}
where $u_0:M\to\R$ is an initial datum, and the unknown $u$ is a real function of variables $(t,x)\in[0,+\infty)\times M$.
 
The first problem that naturally arises is that of finding a good notion of solution. It is in fact well known that such equations do not usually admit classical solutions. For instance, 
the method of characteristics provides,  for smooth initial data and under general assumptions on $H$, classical solutions for \eqref{intro evolutive eq} only for small times, until shocks between characteristics occur. 

A first idea is to look for Lipschitz functions which solve the equation almost everywhere. However, such a notion of solution is inadequate since it lacks of good uniqueness and stability properties, see for instance \cite{barles_book}. That is why two better notions of weak solutions were introduced.\smallskip 

The first one, introduced by Crandall and Lions in 1983 \cite{CL83}, is that of {\em viscosity solution} and it applies to a wide range of first and second order nonlinear PDE. It is in some way reminiscent of distributions, in the sense that it uses test functions to drop derivatives, and permits to name solutions functions that are just continuous, regardless of their regularity. This notion has revealed to be extremely powerful and flexible, and it has been generalized since then in many different directions.\smallskip   

The second one, more geometric in nature, is termed {\em variational solution} and was introduced by Chaperon, Sikorav and Viterbo  (see \cite{chaperon} or \cite{viterbo} and references therein) and is more specific of evolutionary equations of the kind of \eqref{intro evolutive eq}. 
It follows the idea of the method of characteristics, taking the graph of the differential of the initial condition and pushing it with the Hamilton flow. When the image is no longer a graph, a natural way to cut the obtained Lagrangian manifold, known as the graph selector, allows to reconstruct a function which is the desired variational solution.\smallskip 

Note that, as can be expected, both those notions often yield almost everywhere solutions (meaning Lipschitz functions satisfying the Hamilton--Jacobi equation almost everywhere).   In the case of a Hamiltonian $H(x,p)$ which is convex in the momentum variable $p$,  it was proved by Zhukovskaya (see \cite{zhu} and also \cite{beca,wei}) 
that both notions provide the same solution. However, they may differ in the general case. Indeed, variational solutions do not  necessarily  verify the semi--group (or Markov) property. There is a simple way (proposed  by Chaperon) to force variational solutions to verify this property and Wei \cite{wei} recently established that the obtained functions are the viscosity solutions (see also the work of Roos \cite{Roos} on this matter).\smallskip

In this paper we will focus on the solvability of the following system of uncoupled Hamilton--Jacobi equations: 
\begin{equation*}
\begin{cases}
\displaystyle
{\partial_t u}+H(x, d_x u)=0&\quad \hbox{in $(0,+\infty)\times (0,+\infty)\times M$}\medskip\\
\displaystyle {\partial_s u}+G(x, d_x u)=0&\quad \hbox{in $(0,+\infty)\times (0,+\infty)\times M$}\medskip\\
u(0,0,x)=u_0(x)&\quad \hbox{on $M$}.
\end{cases}
\end{equation*}
Such a system is known as {\em multi--time Hamilton--Jacobi equation} and has its roots in economics. There is now a consequent literature concerning this problem, where either one or the other notion of solution has been considered.  The issue about existence and uniqueness of viscosity solutions of the multi--time equation was first addressed by Lions and Rochet \cite{lions}  in the case of convex and coercive Hamiltonians on $\R^{n}\times \R^{n}$ only depending on the momentum variable. This work was subsequently generalized in \cite{barles} to Hamiltonians depending on both variables of $\R^{n}\times \R^{n}$, but still convex and coercive in the momentum. As a counterpart, a commutation hypothesis is necessary (which is always verified in the case of Lions and Rochet), namely the vanishing of the Poisson bracket:
$$\{H,G\}:=\partial_p H\cdot  \partial_x G - \partial_p G \cdot \partial_x H=0\qquad\hbox{in $\R^n\times\R^n$}.$$
Those results were then extended to less regular settings in \cite{MoRa}.

In the framework of symplectic geometry and variational solutions, Cardin and Viterbo \cite{viterbo} extended these results to a larger class of Hamiltonians that do not satisfy any convexity assumption in the momentum variable. The authors proved that the null Poisson bracket condition entails existence and uniqueness of the variational solutions of the multi--time equation.

Knowing the results of Cardin and Viterbo, and the close link between variational and viscosity solutions, it seems natural to ask the following question: if two regular Hamiltonians (non necessarily convex) $H$ and $G$ satisfy the null Poisson bracket condition, can we solve the multi--time Hamilton--Jacobi in the viscosity sense? 

The purpose of this note is to show that the answer to this question is, in general, negative.\smallskip
%

The article is organized as follows. In the first section we recall the definitions of viscosity and variational solutions and we present the link between the two through Wei's result. In the second section we give our main result. More precisely, we show that, when the convexity condition on one   Hamiltonian is dropped, and replaced by a concavity condition, the multi--time equation admits a viscosity solution only for very special choices of the initial datum $u_{0}$ (see Theorem \ref{main} for the precise statement). Given a pair of such Hamiltonians $H$ and $G$, the corresponding set of admissible initial data is typically empty. An example of this fact is provided in Corollary \ref{cor main}.  
 
In the appendix we give the proof of a folklore theorem on existence and uniqueness of viscosity solutions to the evolutionary Hamilton--Jacobi equation. The result is proved under a set of conditions either more general or different than the ones usually assumed in the literature. The techniques and ideas employed are by no means new, however, since they are disseminated in the literature and the adaptation to the case at issue requires some technical work, 
we felt it might be useful to gather them in a single reference and provide a neat proof of this fact. \\

\section{Preliminaries}
\numberwithin{equation}{section}

Throughout the paper, we will call {\em Hamiltonian} a continuous function defined on the cotangent bundle of a smooth and compact Riemannian manifold $M$. We will say that $H$ is {\em coercive} if 
\begin{equation}\label{H coercive}\tag{C}
 \lim_{|p|\to \infty} H(x,p)=+\infty\qquad\hbox{for every $x\in M$.}
\end{equation}
For convenience, we will assume that $M$ is the flat $n$-dimensional torus $\T^n$, however most results keep holding in the more general setting since they are essentially local in nature.\smallskip

In this section, we focus on the case of a single Hamilton--Jacobi equation, i.e.
\begin{equation}\label{hj}\tag{HJ}
 \partial_t u +H(x,D_x u)=0\qquad\hbox{in $(0,+\infty)\times\T^n$,}
\end{equation}
where we assume either $H$ or $-H$ to be a coercive Hamiltonian. 
When $H$ is at least $C^1$ with Lipschitz derivatives ($C^{1,1}$ in short), it induces a Lipschitz vectorfield, called {\em Hamiltonian vectorfield}, defined as follows:
 $$
X_H (x,p) =\big(\partial_pH(x,p),-\partial_x H(x,p)\big)\qquad\hbox{for all $(x,p)\in \T^n \times \R^n$.}
$$
When well defined (which is the case for coercive Hamiltonians), the flow of this vectorfield will be denoted by $(\varphi^t)$.

 \subsection{Viscosity solutions}\label{sez visco}

Let $u$ be a continuous function in $(0,+\infty)\times\T^n$. 

The function $u$ is called a {\em viscosity subsolution} of \eqref{hj} if, for every function $\varphi\in\D{C}^1((0,+\infty)\times\T^n)$ such that $u-\varphi$ attains a local maximum at $(t,x)\in (0,+\infty)\times\T^n$, we have 
\begin{equation}\label{def subsol}
{\partial_t \varphi}(x,t)+H\big(x, D_x \varphi (x,t)\big)\leqslant 0.
\end{equation}
Any such test function $\varphi$ will be called {\em supertangent} to $u$ at $(t,x)$. 

The function $u$ is called a {\em viscosity supersolution} of \eqref{hj} if, for every function $\varphi\in\D{C}^1((0,+\infty)\times\T^n)$ such that $u-\varphi$ attains a local minimum at $(t,x)\in (0,+\infty)\times\T^n$, we have 
\begin{equation}\label{def supersol}
{\partial_t \varphi}(x,t)+H\big(x, D_x \varphi (x,t)\big)\geqslant 0.
\end{equation}
Any such test function $\varphi$ will be called {\em subtangent} to $u$ at $(t,x)$. Last, $u$ is called a {\em viscosity solution} of \eqref{hj} if it is both a sub and a supersolution. 

It is well known, see for instance \cite{barles_book}, that the notions of viscosity sub and supersolutions are local, in the sense that the test function $\varphi$ needs to be defined only in a neighborhood of the point $(t,x)$. Moreover, up to adding to $\varphi$ a quadratic term, such a point can be always assumed to be either a strict maximum or a strict minimum point of $u-\varphi$. In this instance, we will say that $\varphi$ is a {\em strict supertangent} (resp., {\em strict subtangent}) to $u$ at $(t,x)$.

One of the main features of the notion of viscosity solution is that it is extremely stable. Indeed, in rough terms, if $\tilde u$ is a small perturbation of $u$ and the test function $\varphi$ is, for instance, a strict subtangent to $u$ at $(t,x)$, then it will also be a subtangent to $\tilde u$ at a point $(\tilde t,\tilde x)$ close to $(t,x)$. It follows that a continuous function obtained as local uniform limit of a sequence of viscosity solutions of \eqref{hj} is still a solution of the same equation. Furthermore, it has the advantage to be a reasonable notion, in the sense that it yields existence and uniqueness results under mild conditions, as it is illustrated by the following 
now classical result
  
\begin{Th}\label{visco}
Let $H:\T^n\times \R^n\to\R$ such that either $H$ or $-H$ is a coercive Hamiltonian. If $u_0\in \D{C}(\T^n)$, there exists a unique uniformly continuous function $u$ in $[0,+\infty)\times\R^n$ satisfying $u(0,\cdot)=u_0$ and solving \eqref{hj} in the viscosity sense. Furthermore, if $u_0$ is Lipschitz-continuous in $\T^n$, then $u$ is Lipschitz continuous in $[0,+\infty)\times\R^n$ and its Lipschitz constant only depends on $\|Du_0\|_\infty$ and $H$. 
\end{Th}

The proof of this result can be always reduced to the case when the Hamiltonian is {coercive} thanks to the following simple remark:

\begin{pr}\label{prop change sign}
Let $u$ be a continuous function in $(0,+\infty)\times\T^n$. Then $u$ is a subsolution (respectively, a supersolution) of \eqref{hj} if and only if $-u$ is a supersolution (resp., a subsolution) of 
\[
 \partial_t v -H(x,-D_x v)=0\qquad\hbox{in $(0,+\infty)\times\T^n$.}
\]
\end{pr}

Theorem \ref{visco} is known by PDE specialists, however all the the proofs we have found in literature are given under sets of assumptions somewhat different from the ones herein considered. For completeness, we have provided in the Appendix a proof of Theorem \ref{visco} in a slightly more general form, namely we deal with the case when $\T^n$ is replaced by $\R^n$, see Theorem \ref{teo appendix}. Theorem \ref{visco} follows from this by taking the lift of $H$ and of the initial datum $u_0$ to the universal cover of the torus.\medskip 

In the rest of the paper, we will always assume the initial datum $u_0$ to be Lipschitz in $\T^n$ and we will denote by $S^t_H u_0(x)$, or simply by $S^t u_0(x)$ when no ambiguity is possible, the unique Lipschitz solution $u(t,x)$ to the evolutionary Hamilton--Jacobi equation \eqref{hj} with Hamiltonian $H$ and taking the initial datum $u_0$ at $t=0$. 

We end this section by recalling some facts about Tonelli Hamiltonians. We remind that a Hamiltonian $H$ is termed {\em  Tonelli} if it is of class $C^2$, satisfies the following superlinear growth in $p$:
$$\frac{H(x,p)}{|p|}\underset{|p|\to \infty}{\longrightarrow} +\infty,$$
and is strictly convex, in the sense that $\partial^2_{pp} H$ is everywhere positive definite as a quadratic form. We will be interested in some regularity properties of solutions to the Hamilton--Jacobi equation. Recall that a function $f:\T^n\to \R$ is said to be locally semi--concave (resp. semi--convex) (with a linear modulus) if it can be locally written as the sum of a smooth function and a concave (resp. convex) function (see \cite{CanSin} for a complete presentation). In particular, a locally semi--concave function has a supertangent at every point. This implies, for example, that the map $x\mapsto |x|$ is not locally semi--concave because of its singularity at $0$ which is downward. The following well-known fact holds,  (see for example Lemma 4.2 in \cite{BeSur}):
\begin{lm}\label{semi}
Let $H$ be a Tonelli Hamiltonian and $u_0$ an initial datum. Then, for all $t>0$, the function $S^t u_0$ is locally semi--concave in $\T^n$.
\end{lm}
We also underline for further use that a function is both  locally semi--concave and locally semi--convex if and only if it is $C^{1,1}$ ($C^1$ with Lipschitz derivative, see Theorem 6.1.5 in \cite{Fa}).

Using this idea, Patrick Bernard proved the following result (Proposition 4.10 in \cite{BeSur}):
\begin{pr}
Let $u_0$ be a semi--convex function in $\T^n$ and $H$ a Tonelli Hamiltonian. Then, for $t>0$ small enough, $S^t u_0$ is $C^{1,1}$.
\end{pr}

Finally, we will need the following beautiful result due to Fathi (Theorem 6.4.1 in \cite{Fa}):
\begin{Th}\label{transverse}
Let $v: \T^n\to \R$ be a $C^1$ function and $H$ be a Tonelli Hamiltonian whose flow will be denoted $(\varphi^t)$. Let $L$ be the graph of the differential of $v$. If there is an increasing and diverging sequence $(t_n)_{n\in \N}$ such that $\varphi^{t_n}(L)$ is a graph above $\T^n$ for all $n\in \N$, then $L$ is invariant by $(\varphi^t)$. Moreover, it follows that $v$ is $C^{1,1}$ and that there exists a constant $\alpha[0]\in \R$, only depending on $H$, such that $v$ is a strong solution of $H(x,D_x v)=\alpha[0]\quad\hbox{in $\T^n$}$. In particular
$$
\quad S^t v(x) = v(x)-t\alpha[0]\qquad\hbox{for all $(t,x)\in [0,+\infty)\times \T^n$.}
$$
\end{Th}

The constant $\alpha[0]$ appearing above is called {\em critical value} for $H$. It is characterized as the only constant $a$ for which an equation of the kind \quad $H(x,D_x v)=a$\quad in $\T^n$ admits viscosity solutions, see \cite{LPV}. The viscosity solutions of the corresponding critical equation were used by Fathi as central objects in the weak KAM and Aubry-Mather theory (see \cite{Fa,FaSur,BeSur} for introductions), hence they are also known as {\em  weak KAM solutions}.

\subsection{Variational solutions and their link with viscosity solutions}
In this section we give a very brief overview of the notion of variational solutions. For a more detailed account we refer to \cite{beca, wei} and to the references therein. 

We will assume here that $H:\T^n\times \R^n \to \R$ is a $C^2$ Hamiltonian with compact levels, meaning that  $H^{-1}\left(\{a\}\right)$ is compact for all $a\in \R$. This ensures that $H$ has what is called finite propagation speed and that all the notions that will be used are well defined. Moreover, its Hamiltonian flow $(\varphi^t)$ is complete (indeed, $H$ remains constant along the trajectories of the flow).

Let us consider a smooth function $u_0:\T^n \to \R$, and we will denote by $L_{u_0}\subset \T^n\times \R^n$ the graph of the differential of $u_0$. It is a smooth Lagrangian submanifold of $\T^n\times \R^n$ for its standard symplectic structure. As it is commonly done, we introduce the Hamiltonian $$\HH : \R\times \T^n\times \R\times \R^n \to \R ,\quad (t,x,\tau,p)\mapsto \tau+H(x,p),$$
so that solving the evolutionary Hamilton--Jacobi equation is equivalent to solving $\HH(t,x,D_{(t,x)}U)=0$. We also define 
$$\Gamma_{u_0}=\{(0,x,-H(x,p),p),\ \ (x,p)\in L_{u_0}\}.$$
Let us denote by $(\Phi^t)$ the Hamiltonian flow generated by $X_\HH = (\partial_{(\tau,p)} \HH,-\partial_{(t,x)} \HH)$. For every fixed $T>0$ we set 
$$L_{\HH,u_0}=\bigcup_{t\in [0,T]} \Phi^t(\Gamma_{u_0}).$$
It is then known that $L_{\HH,u_0}$ admits a {\em generating function quadratic at infinity}, that is a function 
$$S:[0,T]\times \T^n \times \R^k \to \R, \quad (t,x,\eta)\mapsto S(t,x,\eta),$$
for some integer $k$, such that $S$ coincides with a quadratic form outside a compact set, $0$ is a regular value of the mapping $(t,x,\eta)\mapsto \partial_\eta S(t,x,\eta)$ and 
$$L_{\HH,u_0}=\big\{\big(t,x,\partial_{(t,x)}S(t,x,\eta)\big )\,:\, \partial_\eta S (t,x,\eta)=0 \big\}.$$ 
Moreover, the construction of $S$ is done in such a way that if $ \partial_\eta S(0,x,\eta)=0$ then $S(0,x,\eta)=u_0(x)$. 
Then, define 
$$V^tu_0(x) := \inf_{[\sigma]=A} \max_{\eta\in\sigma} S(t,x,\eta),$$ 
where $A$ generates some homology group depending on the topology of the sublevels of $S$. The function $(t,x)\mapsto V^t u_0(x) $ is the variational solution. This construction can also be generalized to Lipschitz initial data by limit arguments. 

If $H$ is convex in $p$, then $S^tu_0(x)=V^tu_0(x)$ for any Lipschitz continuous initial datum $u_0$. However, for general  Hamiltonians, this is not always the case since $V^t\circ V^s u_0$ may differ from $V^{t+s}u_0$. This is what is commonly referred to as the lack of the {\em semi--group property}. Recently, Wei \cite{wei,Wei2} established the following link (which is formulated in a more general way in her thesis):
\begin{Th}[\cite{wei}]\label{wei}
For any Lipschitz function $u_0:\T^n\to \R$, the following holds:
$$ 
S^tu_0(x) = \lim_{n\to +\infty} (V^{t/n})^n u_0(x)\qquad\hbox{for all $(t,x) \in (0,+\infty)\times \T^n$.}
$$
\end{Th}
\medskip

\section{The multi--time Hamilton--Jacobi equation}
We will only be interested in systems of two equations. Therefore, in the following, $H$ and $G$ will be two smooth and proper Hamiltonians on $\T^n\times \R^n$. Note that no convexity assumptions are taken in the momentum variable. Given a Lipschitz function $u_0 : \T^n \to \R$, we are looking for a real function $u(t,s,x)$ defined on $[0,+\infty)\times [0,+\infty)\times \T^n$ which solves 
\begin{equation}\label{multi--time HJ}
\begin{cases}
\displaystyle
{\partial_t u}+H(x, D_x u)=0&\quad \hbox{in $(0,+\infty)\times (0,+\infty)\times \T^n$}\medskip\\
\displaystyle {\partial_s u}+G(x, D_x u)=0&\quad \hbox{in $(0,+\infty)\times (0,+\infty)\times \T^n$}\medskip\\
u(0,0,x)=u(x)&\quad \hbox{on $\T^n$}
\end{cases}
\end{equation}
in the viscosity sense. 
As first remarked in \cite{barles}, a necessary condition for viscosity solutions of \eqref{multi--time HJ} to exist, for any Lipschitz initial datum, is that $H$ and $G$ commute, in the  sense that their Poisson bracket vanishes, i.e.
$$
\{H,G\}:=\partial_p H\cdot  \partial_x G - \partial_p G \cdot \partial_x H=0\qquad\hbox{in $\T^n\times\R^n$}.
$$
This can be derived by differentiating the equation and by using the method of characteristics (for smooth initial data), see for instance Appendix C in \cite{davzav}. When the Hamiltonians are convex and coercive in the momentum variable,  this  condition is sufficient as well, as it is proved in \cite{barles}. In the more general case when the Hamiltonians are proper but not necessarily convex in the momentum, it is proved in \cite{viterbo} that the null Poisson bracket condition implies the existence of variational solutions to the multi--time equation.    
Knowing the links between variational and viscosity solutions (Theorem \ref{wei}), it seems natural to wonder if the vanishing of $\{H,G\}$ also implies the existence of solutions to (\ref{multi--time HJ}) in the viscosity sense. Using the semi--group property and Theorem \ref{visco}, it is easily seen that this is equivalent to requiring that, for any Lipschitz  $u_0 : \T^n \to \R$, the following hold:
$$ 
u(t,s,\cdot)=S^t_H\circ S^s_G u_0 = S^s_G\circ S^t_H u_0\qquad\hbox{for every $s,t>0$}.
$$
 
We will prove the following theorem:
 \begin{Th}\label{main}
 Let $H$ and $G$ be two commuting  Hamiltonians on $\T^n\times \R^n$ such that both $H$ and $-G$ are Tonelli. Then there exists a solution to the multi--time equation (\ref{multi--time HJ}) with Lipschitz initial datum $u_0:\T^n\to \R$ if and only if $u_0$ is $C^{1,1}$ and the graph of its differential is invariant by both the Hamiltonian flows of $H$ and $G$. Moreover, in this case, there are two constants $c_H$ and $c_G$ such that 
 $$H(x,D_xu_0)=c_H \quad {\rm and }\quad G(x,D_xu_0)=c_G, \qquad\hbox{for every $x\in \T^n$},$$
which means that 
$$ 
S^t_H\circ S^s_G u_0 = S^s_G\circ S^t_H u_0 = u_0-tc_H-s\, c_G\qquad\hbox{for every $s,t>0$}.
$$
 \end{Th} 

\begin{oss}
Note that $c_H=\alpha_H[0]$ and $c_G=-\alpha_{-G}[0]$, where $\alpha_H[0]$ and $\alpha_{-G}[0]$ are the critical values associated with $H$ and $-G$.
\end{oss}

Before proving this theorem, we establish a simple lemma:
 \begin{lm}\label{check}
Let $G:\T^{n}\times\R^{n}\to\R$ such that either $G$ or $-G$ is a coercive Hamiltonian.  Let us set 
$$
\bar G(x,p) := -G(x,-p)\qquad\hbox{for every $(x,p) \in\T^n\times \R^n$}.
$$
Then, for any Lipschitz function $u_0 : \T^n\to \R$, the following holds :
$$
S^s_G u_0 = -S^s_{\bar G} (-u_0)\qquad\hbox{for every $s>0$}.
$$
In particular, if $-G$ is Tonelli, then $S^s_G u$ is locally semi--convex for any $s>0$. 
 \end{lm}
 
\begin{proof}
The first assertion follows from Proposition \ref{prop change sign}, while the second is a consequence of Lemma \ref{semi}. 
\end{proof}
In the next proof, recall that if $w:\T^n \to \R$ is any $C^1$ function, we denote by $L_w\subset \T^n \times \R^n$ the graph of its differential.
\begin{proof}[Proof of Theorem \ref{main}]
Let $u_0$ be an initial condition for which there exists a solution to the multi--time equation. We have seen that in this case 
$$
S^s_G\circ S^t_H u_0 =  S^t_H\circ S^s_G u_0 \qquad\hbox{for every $s,t>0$}.
$$
In particular, as the left hand side is locally semi--convex by Lemma \ref{check} and the right hand side is locally semi--concave, by Lemma \ref{semi}, we deduce that for all $s,t>0$, the function $S^s_G\circ S^t_H u_0 =  S^t_H\circ S^s_G u_0$ is $C^{1,1}$. 

Let $s>0$ be fixed. For $\e>0$, set $v_\e=S^s_G\circ S^\e_H u_0 \in C^{1,1}$. It follows from \cite[Proposition 4.11.1 and Theorem 6.2.2]{Fa} that 
$$
\varphi^t_H(L_{v_\e}) = L_{S^t_H v_\e}\qquad\hbox{for every $t>0$}.
$$
In particular, we may
 let $t$ go to infinity and use Fathi's theorem \ref{transverse}, yielding the existence of $c_H$ such that 
$$
S^t_H v_\e=  v_\e -tc_H\qquad\hbox{for every $t\geqslant 0$}.
$$
Now letting $\e\to 0^{+}$ and passing to the limit, we deduce that 
\begin{equation}\label{eq1 main thm}
S^t_H \circ S^s_G u_0=  S^s_G u_0 -tc_H\qquad\hbox{for every $t\geqslant 0$}.
\end{equation}
By continuity, the above identity holds for $s=0$ as well.

Let us now fix $t>0$. The same argument applied to $\bar H$, $\bar G$ and $-u_{0}$ yields the existence of a constance $c_{G}$ such that  
\begin{equation}\label{eq2 main thm}
S^s_G \circ S^t_H u_0=  S^t_H u_0 -sc_G\qquad\hbox{for every $s\geqslant 0$}.
\end{equation}
By continuity, the above identity holds for $t=0$ as well. Putting \eqref{eq1 main thm} and \eqref{eq2 main thm} together we get
\[
S^t_H \circ S^s_G u_0=u_0 -sc_{G}-tc_H=S^s_G \circ S^t_H u_0\qquad\hbox{for every $s,t\geqslant 0$},
\]
finally showing that $u_0$ is $C^{1,1}$ and a strong solution of the stationary equations associated with $H$ and with $G$. 
\end{proof}

Even if for a Tonelli Hamiltonian $H$ there are always weak KAM solutions, they may all fail to be $C^{1,1}$. In this instance, the previous theorem implies that there are no global solutions to the multi--time equation for $H$ and $-H$ (nor in fact for $H$ and $-G$ for any Tonelli Hamiltonian $G$). 

We furnish an example to conclude.

\begin{co}\label{cor main}
There exists a Tonelli Hamiltonian $H:\T\times\R\to\R$ such that the multi--time equation \eqref{multi--time HJ} with $G:=-H$ does not admit viscosity solutions, for any Lipschitz initial datum $u_0$.
\end{co}

\begin{proof}
According to Theorem \ref{main}, we just need to exhibit a Tonelli Hamiltonian $H$ that does not admit weak KAM solutions of class $C^{1,1}$. The example that follows is classical, see for instance \cite[Section 4.14]{Fa}. 

Consider the pendulum equation, corresponding to 
$$H(x,p)=\frac 12 |p|^2+ \cos (2\pi x) ,\  \mathrm{on}\  \T^1\times \R.$$
The Hamiltonian vector-field is then given by $X_H(x,p)= \big(p,2\pi\sin(2\pi x)\big)$. It possesses  $2$ fixed points, $(\frac 12 , 0)$ which is an elliptic fixed point, and $(0,0)$ which is hyperbolic. The stable and unstable manifold are given by the two separatrices which equations are $p=\pm \sqrt{2\big(1-\cos(2\pi x)\big)}=\pm 2\sin(\pi x)$. Those separatrices bound an elliptic island filled with periodic trajectories. In this case, it is known that $\alpha[0]= 1$. Moreover, the only weak KAM solution (up to addition of a constant) is given by 
$$
u(x) = 
\left\{ 
\begin{aligned}
 &\int_0^x 2\sin(\pi t) dt = \frac{2}{\pi}\big(1-\cos(\pi x)\big) , \quad & \forall x\in \big[0,\frac 12\big], \\
 & \frac{2}{\pi}- \int_{\frac{1}{2}}^x 2\sin(\pi t) dt = \frac{2}{\pi}\big(1+\cos(\pi x)\big) ,\quad & \forall x\in \big(\frac 12, 1\big),
\end{aligned}
\right.
$$
which is not of class $C^{1,1}$. 
Indeed,  the graph of the derivative of any other weak KAM solution $v$ must stay on the union of the two separatrices, the fact that $v$ is semi--concave means that this graph may only "jump downwards", therefore it may possess only one discontinuity, and in order for $v$ to  be $1$-periodic, this jump must occur at $1/2$. Hence $v'=u'$, that is $v-u$ is constant. 
 \end{proof}

It is proved in \cite{ZA3,chinois} that two commuting Tonelli Hamiltonians have the same weak KAM solutions, therefore  the previous theorem is consistent with this fact, in the sense that it is normal that whenever $H$ has a $C^{1,1}$ weak KAM solution, it will be the same for $G$.
 
 Let us conclude with one last speculation. The previous results, in order to be global, are stated for two Tonelli Hamiltonians. However, when a single Hamiltonian is  convex in the fibers in some areas (say for $x$ in an open set $U$ of the torus), and concave in others (say for $x$ in an open set $V$ of the torus), the previous discussions show that 
variational solutions of the evolutionary equation will be  either semi--concave (resp. semi--convex) in $U$ (resp. $V$)  
 for small times, i.e. as long as the method of characteristics applies and the latter do not leave those open sets. This need not remain true  in the long time. 
On the other hand, due to the semi--group property and the finite speed propagation, the viscosity solutions always remain semi--concave in $U$ (resp. semi--convex in $V$). These two different kinds of regularity might give a new insight as to detecting when 
variational solutions differ from viscosity solutions and thereby do not verify the semi--group property. We also mention the recent work \cite{bernardnew} on very similar issues.
 
%

\appendix
\section{}

In this Appendix, we want to give a proof of Theorem \ref{visco} in a slightly more general setting. In what follows we will denote by $\D{UC}(X)$ and $\D{BUC}(X)$ the space of uniformly continuous and bounded uniformly continuous real functions on the metric space $X$, respectively.\smallskip

We consider the following evolutionary Hamilton--Jacobi equation:
\begin{equation}\label{appendix eq hj}\tag{A}
{\partial_t u}+F(x, D_x u)=0\quad \hbox{in $(0,+\infty)\times \R^n$,}
\end{equation}
where the Hamiltonian $F:\R^n\times\R^n\to \R$ satisfies the following conditions:
\begin{itemize}
\item[(F1)] \quad $F\in\D{BUC}(\R^n\times B_R)$\ \ \qquad for every 
    $R>0$;\smallskip
\item[(F2)] \quad $\displaystyle{\inf_{x\in\R^n}{F(x,p)}\to +\infty}\quad$\quad\  as $|p|\to +\infty$.\medskip
\end{itemize}

We will prove the following result:

\begin{Th}\label{teo appendix}
Let $F:\R^n\times\R^n\to\R$ be a Hamiltonian satisfying (F1)-(F2) and let $u_0\in \D{UC}(\R^n)$. Then there exists a unique function $u\in\D{UC}([0,+\infty)\times\R^n)$ with $u(0,\cdot)=u_0$ that solves \eqref{appendix eq hj} in the viscosity sense. If $u_0$ is Lipschitz in $\R^n$, then $u$ is Lipschitz continuous in $[0,+\infty)\times\R^n$ and its Lipschitz constant only depends on $\|Du_0\|_\infty$ and $F$. 
\end{Th}

We need to generalize the notions of viscosity sub and supersolution given at the beginning of Section \ref{sez visco} to the case of discontinuous functions. Let $\Omega$ be an open subset of $(0,+\infty)\times\R^n$ and let $u:\Omega\to \R$ be a possibly discontinuous locally bounded function. We define the {\em upper} and {\em lower semicontinuous envelope} of $u$, denoted by $u^*$ and $u_*$ respectively, as follows:
\[
 u^*(t,x):=\inf_{r>0}\sup_{(s,y)\in B_r\left((t,x)\right)\cap \Omega} u(s,y),
 \qquad 
 u_*(t,x):=\sup_{r>0}\inf_{(s,y)\in B_r\left((t,x)\right)\cap \Omega} u(s,y), 
\]
for every $(t,x)\in\Omega$. 
We say that $u^*$ is a {\em (non--continuous) viscosity subsolution} of \eqref{appendix eq hj} in $\Omega$ if \eqref{def subsol} holds for any supertangent $\varphi$ to $u^*$ at the point $(t,x)\in\Omega$. The gradient of any such supertangent $\varphi$ at the point $(t,x)$ is called {\em supergradient} of $u^*$ at $(t,x)$. We will denote by $D^+u^*(t,x)$ the set consisting of all supergradients of $u^*$ at $(t,x)$.  Analogously, we say that 
$u_*$ is a supersolution of \eqref{appendix eq hj} in $\Omega$ if \eqref{def supersol} holds for any subtangent $\varphi$ to $u_*$ at the point $(t,x)\in\Omega$. 

Given a function $f:\ccyl\to\R$, we introduce the following inequality depending on a pair of constants $\alpha>0$ and $M_\alpha>0$:
\begin{equation}\label{condition}
f(t,x)\geqslant f(0,x)- (\alpha+M_\alpha t)\qquad\hbox{for all $(t,x)\in\ccyl$.}
\tag{C}
\end{equation}

We will need the following comparison result between discontinuous sub and supersolutions.

\begin{pr}\label{prop appendix comparison}
Let $v$ and $u$ be, respectively, an upper and lower semicontinuous function on $[0,+\infty)\times\R^n$ satisfying condition \eqref{condition} for a pair of constants $\hat\alpha>0$ and $M_{\hat\alpha}>0$. 
Let us assume that $v$ is a subsolution and $u$ is a supersolution of \eqref{appendix eq hj}, and that either one of the following conditions hold:
\begin{itemize}
\item[(a)] for every $\alpha>0$, there exists $M_\alpha>0$ such that  \eqref{condition} holds with $f:=v$;\smallskip
\item[(b)] $v$ is continuous on $\ccyl$. 
\end{itemize}
Then
\[
 \sup_{[0,+\infty)\times\R^n}\big(v-u\big) 
  \leqslant
  \sup_{\R^n}\big(v(0,\cdot)-u(0,\cdot)\big). 
\]
\end{pr}

\begin{oss}\label{oss appendix comparison}
If the supersolution $u$ does not satisfy condition \eqref{condition}, then the comparison principle stated above does not hold, no matter how regular the subsolution $v$ is, see for instance Example 4.9 in \cite{CaSi}. We remark that condition \eqref{condition} is always fulfilled by $u$ and $v$ when the latter are absolutely continuous in $\ccyl$. In particular, Proposition \ref{prop appendix comparison} entails uniqueness of the solution of \eqref{appendix eq hj} in $\D{UC}\left(\ccyl\right)$. 
\end{oss}

For the proof of Proposition \ref{prop appendix comparison}, we will need to regularize the subsolution  $v$ through the following {\em sup--convolution in time}, see \cite{bardi}: for every fixed $\delta>0$, define 
\begin{equation}\label{def appendix v^delta}
v^\delta(t,x):=\sup_{s\geqslant 0} \left\{ v(s,x)-\frac{|t-s|}{\delta}\, \right\}, \qquad\hbox{$(t,x)\in [0,+\infty)\times\R^n$.}
\end{equation}
For every $(t,x)\in [0,+\infty)\times\R^n$, we denote by $\M^\delta(t,x)$ the set of $s\geqslant 0$ that realize the supremum in the definition of $v^\delta(t,x)$. We set
\[
 \Omega^\delta:=\D{int}\left\{(t,x)\in\cyl\,:\,v^\delta(t,x)>v(0,x)-\frac{t}{\delta},\right\}, 
\]
that is, $\Omega^\delta$ is the largest open subset of $(0,+\infty)\times\R^n$ made up of points $(t,x)$
for which $0\not\in \M^\delta(t,x)$.
Last, we set $\mu:=\max\{1,\sup_{\R^n\times\R^n} -F\}$, which is finite by assumption (F2). The following holds:

\begin{pr}\label{prop v^delta}
Let us assume that $v$ is upper semicontinuous on $\ccyl$ and a non-continuous subsolution of \eqref{appendix eq hj}. Then for every $\delta<{1}/{\mu}$ we have:
\begin{itemize}
 \item[{\em (i)}] \quad $v(t,x)\leqslant v^\delta(t,x) \leqslant v(0,x)+\mu t$\qquad for every $(t,x)\in\ccyl$. \smallskip
  \item[{\em (ii)}] For every $(t,x)\in\ccyl$ the set $\M^\delta(t,x)$ is nonempty and contained in $[0,2t/(1-\delta\mu)]$. In particular, $v^\delta$ is upper semicontinuous in $\ccyl$.\smallskip
  \item[{\em (iii)}] If $(t_0,x_0)\in \Omega^\delta$, then 
\[
 D^+ v^\delta(t_0,x_0)\subseteq D^+ v(s_0,x_0)\qquad\hbox{for every $s_0\in \M^\delta(t_0,x_0)$.}
\]
In particular, $v^\delta$ is an upper semicontinuous subsolution of \eqref{appendix eq hj} in $\Omega^\delta$.\smallskip
  \item[{\em (iv)}] The function $v^\delta(\cdot,x)$ is $1/\delta$--Lipschitz continuous in $[0,+\infty)$ for every fixed $x\in\R^n$. Moreover, there exists a positive constant  $L_\delta$ such that  
$$
|D_x v^\delta(t,x)|\leqslant L_\delta\qquad\hbox{for a.e. $(t,x)\in \Omega^\delta$.} 
$$
  \item[{\em (v)}] Let us assume that $v$ satisfies \eqref{condition} for some positive constants $\alpha$ and $M_\alpha$. Then for every $\delta>0$ small enough   
\[
 \left[2\alpha\,\frac{\delta}{1-\delta M_\alpha},+\infty\right)\times\R^n\subset \Omega^\delta
\]
\end{itemize}
\end{pr}

\begin{proof}
{\em (i)} The first inequality is obvious. To prove the second inequality, we first observe that $v$ satisfies 
\[
 \partial_t v \leqslant -F(x,Dv) \leqslant \mu\qquad\hbox{in $\cyl$}
\]
in the viscosity sense. We infer  
\[
 v(s,x)\leqslant v(\tau,x)+\mu(s-\tau)\qquad \hbox{for every $x\in\R^n$ and $0<\tau<s$,}
\]
in particular, by letting $\tau\to 0^+$ and by using the upper semicontinuity of $v$, we get 
\begin{equation}\label{appendix ineq v}
v(s,x) \leqslant v(0,x)+\mu s\qquad \hbox{for every $(s,x)\in\ccyl$.}
\end{equation}
For $\delta <1/\mu$ we get 
\[
 v^\delta(t,x)\leqslant v(0,x)+\sup_{s\geqslant  0}\left\{\mu s-\frac{|t-s|}{\delta}\,\right\}=v(0,x)+\mu t.
\]

{\em (ii)} Let us fix $(t,x)\in\ccyl$. The fact that $\M^\delta(t,x)$ is nonempty follows from \eqref{appendix ineq v} and the upper semicontinuity of $v$. Pick $s\in\M^\delta(t,x)$. From the definition of $v^\delta(t,x)$ and \eqref{appendix ineq v} we have 
\[
 v(0,x)-\frac{t}{\delta}\leqslant v^\delta(t,x)=v(s,x)-\frac{|s-t|}{\delta}\leqslant v(0,x)+\mu s-\frac{|s-t|}{\delta}.
\]
A straightforward computation shows that $s\leqslant 2t/(1-\delta\mu)$. The upper semicontinuity of $v^\delta$ readily follows from this and from the upper semicontinuity of $v$.  

{\em (iii)} First note that $(t_0,x_0)\in \Omega^\delta$ implies $\M^\delta(t_0,x_0)\subset (0,+\infty)$. To prove the first assertion, it suffices to observe that, if $\varphi$ is a supertangent to $v^\delta$ at $(t_0,x_0)\in\Omega^\delta$, then the function $(t,x)\mapsto \varphi(t+(t_0-s_0),x)$ is a supertangent to $v$ at $(s_0,x_0)$ for every fixed $s_0\in\M^\delta(t_0,x_0)$. Since $v^\delta$ is upper semicontinuous by {\em (ii)} and $F$ is time--independent, we infer that $v^\delta$ is an upper semicontinuous subsolution of \eqref{appendix eq hj} in $\Omega^\delta$.

{\em (iv)} The fact $v^\delta$ is $1/\delta$--Lipschitz in $t$ is apparent by definition. By using {\em (iii)} we get
\[
 F(x,Dv^\delta)\leqslant -\partial_t v^\delta\leqslant \frac{1}{\delta}\qquad\hbox{in $\Omega^\delta$,}
\]
and the estimate on $|D_x v^\delta|$ in $\Omega^\delta$ follows from the coercivity of $F$, see Lemma 2.5 in \cite{barles_book}.

{\em (v)} Pick $\delta<1/M_\alpha$ and set $T_\delta:=2\delta\alpha/(1-\delta M_\alpha)$. Then every $(t,x)\in (T_\delta/2,+\infty)\times\R^n$ satisfies  
\[
v(0,x)-\frac{t}{\delta}<v(0,x)-(M_\alpha t+\alpha)\leqslant v(t,x)\leqslant v^\delta(t,x),
\]
in particular $[T_\delta,+\infty)\times\R^n\subset\Omega^\delta$, as it was to be shown. 
\end{proof}

\noindent{\em Proof of Proposition \ref{prop appendix comparison}.} We assume $\sup_{\R^n}\big(v(0,\cdot)-u(0,\cdot)\big)<+\infty$, being the statement otherwise trivial. Furthermore, up to adding a suitable constant to $v$, we can assume that $\sup_{\R^n}\big(v(0,\cdot)-u(0,\cdot)\big)=0$. The assertion is thus reduced to proving that $v-u\leqslant 0$ in $\ccyl$. The proof is divided in two steps.\smallskip

\noindent{\em Step 1.} We first assume $v$ Lipschitz continuous in $\ccyl$, i.e. there exists $L>0$ such that
\[
 |v(t,x)-v(s,y)|\leqslant L(|x-y|+|t-s|)\qquad\hbox{for every $(t,x), (s,y)\in\ccyl$.}
\]
Let us suppose by contradiction that there exists $(\hat t,\hat x)\in\Tcyl$, for some $T>0$, such that
\[
 v(\hat t,\hat x)-u(\hat t,\hat x)=\gamma>0.
\]
Let us introduce a pair of functions  $\cchi,\,\Phi:\left(\Tlccyl\right)^2\to\R$ defined as 
\begin{eqnarray*}
 \cchi(t,x,s,y)&:=& \frac{1}{2\eps}\big(|x-y|^2+|t-s|^2\big)+\beta\langle x\rangle+\dfrac{\eta}{T-t},\\
 \Phi(t,x,s,y)&:=& v(t,x)-u(s,y)-\cchi(t,x,s,y),
\end{eqnarray*}
where $\langle x\rangle:=\sqrt{1+|x|^2}$, $\eta$ is a positive constant that will be suitably chosen, and $\eps$, $\beta$ are positive parameters that will be sent to $0$.    
We choose $\eta>0$ and $\beta>0$ small enough such that 
\[
 \beta\langle\hat x\rangle+\dfrac{\eta}{T-\hat t}\leqslant \frac{\gamma}{2}.
\]
Then 
\[
\sup_{\left(\Tlccyl\right)^2} \Phi \geqslant \Phi(\hat t,\hat x,\hat t,\hat x)
=
\gamma- \left(\beta\langle\hat x\rangle+\dfrac{\eta}{T-\hat t}\right)
\geqslant
\frac{\gamma}{2}
\]
We proceed to show that such a supremum is attained. By hypothesis, 
\[
 u(s,y)\geqslant u(0,y)-({\hat\alpha}+M_{\hat\alpha}s)\qquad\hbox{for every $(s,y)\in\ccyl$}
\]
for a pair of constants ${\hat\alpha}>0$ and $M_{\hat\alpha}>0$. 
Then
\begin{eqnarray*}
\Phi(t,x,s,y)
\leqslant 
L(t+|x-y|)+({\hat\alpha}+M_{\hat\alpha}s)
-
\frac{1}{2\eps}\left(|x-y|^2+|t-s|^2\right)-\beta\langle x\rangle-\dfrac{\eta}{T-t}
\end{eqnarray*}
If $(t,x,s,y)\in \Tlccyl$ satisfies $\Phi(t,x,s,y)>0$, then 
\begin{equation}\label{eq appendix pre-max}
\left(\dfrac{|x-y|}{2\eps}-L\right)|x-y|+\dfrac{|t-s|^2}{2\eps}
+
\beta\langle x \rangle + \dfrac{\eta}{T-t}
\leqslant 
{\hat\alpha} +(M_{\hat\alpha}+L)T=:C.
\end{equation}
This readily implies that there exists $(t_\eps, x_\eps, s_\eps, y_\eps)\in (\Tlccyl)^2$ such that 
\begin{equation}\label{eq appendix max}
\Phi(  t_\eps,  x_\eps,  s_\eps,  y_\eps)=\sup_{(\Tlccyl)^2}\Phi\geqslant \frac{\gamma}{2}
\end{equation}
This  maximum point also depends on the parameters $\eta$ and $\beta$, but such dependence is omitted to ease notations. We infer from \eqref{eq appendix pre-max} that 
\[
\langle x_\eps\rangle \leqslant \dfrac{C}{\beta},
\quad
t_\eps\leqslant T-\dfrac{\eta}{C}.
\]
By Proposition 3.7 in \cite{users}, we also know that 
\[
\dfrac{|t_\eps-s_\eps|^2}{\eps}\to 0,
\quad
\dfrac{|x_\eps-y_\eps|^2}{\eps}\to 0
\qquad
\hbox{as $\eps\to 0^+$}.
\]
We infer that $(t_\eps,x_\eps)$ and $(s_\eps,y_\eps)$ converge, along a subsequence as $\eps\to 0^+$, to a point 
$(t_0,x_0)\in [0,T-\eta/C]\times\R^n$ satisfying, cf. Proposition 3.7 in \cite{users}, 
\[
v(t_0,x_0)-u(t_0,x_0)-\beta\langle x_0 \rangle - \dfrac{\eta}{T-t_0}
=
\Phi(t_0,x_0,t_0,x_0)=\sup_{(t,x)\in\Tlccyl} \Phi(t,x,t,x)
\geqslant 
\dfrac{\gamma}{2}.
\]
Since $v(0,\cdot)-u(0,\cdot)\leqslant 0$ on $\R^n$,  we derive that $t_0>0$, hence  the points $(t_{\eps},x_{\eps})$, $(s_\eps,y_\eps)$ definitively  belong to an open neighborhood $I\times U$ of $(t_0,x_0)$ compactly contained in $(0,T)\times\R^n$.

We can now exploit the fact that $\varphi(t,x):=\cchi(t,x,  s_\eps,  y_\eps)$ is  
 a supertangent to $v$ at $(  t_\eps,  x_\eps)\in\Tcyl$ and  
$\psi(s,y):=-\cchi(  t_\eps,  x_\eps, s, y)$ is a subtangent to $u$ at $(  s_\eps,  y_\eps)\in\Tcyl$. 
Since $v$ is an upper semicontinuous subsolution and $u$ is a lower semicontinuous supersolution of \eqref{appendix eq hj}, we get
\begin{equation}\label{eq appendix comparison}
\partial_t\varphi(  t_\eps,  x_\eps)-\partial_s\psi(  s_\eps,  y_\eps)
\leqslant
F\big(  y_\eps,D\psi(  s_\eps,  y_\eps)\big)-F\big(  x_\eps,D\varphi(  t_\eps,  x_\eps)\big)
\end{equation}
An easy computation shows that
\[
\partial_t\varphi(  t_\eps,  x_\eps)-\partial_s\psi(  s_\eps,  y_\eps)
=
\dfrac{\eta}{(T-t_\eps)^2}
>
\dfrac{\eta}{T^2},
\]
while 
\[
 |D\varphi(  t_\eps,  x_\eps)-D\psi(  s_\eps,  y_\eps)|\leqslant \beta.
\]
Moreover, by the fact that $v$ is $L$--Lipschitz and $\varphi$ is a supertangent to $v$ at 
$(  t_\eps,  x_\eps)$, we get
\[
\left| D\varphi(  t_\eps,  x_\eps)\right|
=
\left | D\psi(  s_\eps,  y_\eps) +\beta\frac{  x_\eps}{\langle   x_\eps \rangle}\right|\leqslant L,
\]
yielding in particular that $| D\psi(  s_\eps,  y_\eps) |\leqslant L+\beta$. We now use the fact that $F$ is uniformly continuous in 
$U\times B_{L+1}$ in view of (F1): by sending $\eps$ and $\beta$ to $0$ in \eqref{eq appendix comparison} 
and by making use of the above estimates, we get 
$\eta/T^2\leqslant 0$, in contrast with the choice of $\eta>0$.\medskip

\noindent{\em Step 2.} Let us prove the result in the general case. Let  $\alpha>0$ and  $M_\alpha>0$ such that condition  \eqref{condition} holds with $f:=v$. Set 
$T_\delta:=2\delta\alpha/(1-\delta M_\alpha)$. According to Proposition \ref{prop v^delta}, for $\delta>0$ small enough the function $(t,x)\mapsto v^\delta(T_\delta+t,x)$ is Lipschitz continuous in 
$\ccyl$ and is a subsolution of \eqref{appendix eq hj} in $\cyl$, hence we can apply Step 1 to infer 
\begin{equation}\label{eq end comparison}
v^\delta(T_\delta+t,x)-u(t,x)
\leqslant 
\sup_{\R^n} \left(v^\delta(T_\delta,\cdot)-u(0,\cdot)\right)
\end{equation}
for any fixed $(t,x)\in \cyl$. 
If condition (a) holds, namely if the constant $\alpha>0$ in \eqref{condition} with $f:=v$ can be chosen arbitrarily small, we argue as follows:  by Proposition \ref{prop v^delta} we know that 
\[
v(t,x)-\frac{T_\delta}{\delta} \leqslant v^\delta(t,x)-\frac{T_\delta}{\delta} \leqslant v^\delta(T_\delta+t,x) \leqslant v(0,x)+\mu(T_\delta+t), 
\]
so
\[
v(t,x)-u(t,x)
\leqslant
\sup_{\R^n} \big(v(0,\cdot)-u(0,\cdot)\big)
+
\frac{2\alpha}{1-\delta M_\alpha}(1+\mu\delta).
\]
The assertion follows by sending $\delta\to 0^+$ and since $\alpha>0$ was arbitrarily chosen. 

If condition (b) holds, namely if $v$ is continuous in $\ccyl$, we exploit Proposition \ref{prop v^delta}-(i) to derive from  \eqref{eq end comparison}  that
\[
v(T_\delta+t,x)-u(t,x)
\leqslant 
\sup_{\R^n} \big(v(0,\cdot)-u(0,\cdot)\big)
+
\mu T_\delta.
\]
The assertion follows by sending $\delta\to 0^+$ since $T_\delta\to 0$ and $v$ is continuous on $\ccyl$. 

\qed

\medskip
The proof of Theorem \ref{teo appendix} is standard. It follows from Proposition \ref{prop appendix comparison} by making use of Perron's method.
\smallskip 

{\em Proof of Theorem \ref{teo appendix}.} We only need to show the existence of a solution, for the uniqueness comes from Proposition \ref{prop appendix comparison}.

Let us first assume $u_0\in\D{Lip}(\R^n)$. In view of the assumptions (F1)-(F2), there exists a positive constant $M$ such that 
$M>\sup\{|F(x,p)|\,:\,|p|\leqslant \|Du_0\|_\infty\,\}$. Then the functions 
\[
\underline u(t,x):=u_0(x)-Mt,\qquad \overline u(t,x):=u_0(x)+Mt,\qquad (t,x)\in\ccyl
\]
are, respectively, a sub and a supersolution to \eqref{appendix eq hj}. We denote by $\Sub$ the family of continuous functions on $\ccyl$ that are subsolutions of \eqref{appendix eq hj} and we set, for every $(t,x)\in\ccyl$,   
\[
u(t,x)=\sup\left\{v(t,x)\,:\,v\in\Sub,\ \underline u\leqslant v \leqslant \overline u\quad \hbox{on $\ccyl$}\right\}.
\]                                                                                                                    
Then $u^*$ and $u_*$ are, respectively, a sub and a supersolution of \eqref{appendix eq hj}, see Theorem 2.14, Chapter V in \cite{bardi}, and satisfy 
\[
u_0(x)-Mt\leqslant u_*(t,x) \leqslant u^*(t,x) \leqslant u_0(x)+Mt \qquad \hbox{on $\ccyl$}.
\]
We can therefore apply the comparison principle stated in Proposition \ref{prop appendix comparison} to infer that  $u^*\leqslant u_*$ on $\ccyl$. Since the other inequality is obvious by definition, we get that $u$ is continuous on $\ccyl$, solves \eqref{appendix eq hj} in the viscosity sense and takes the initial datum $u_0$ at $t=0$. 
It is left to show the Lipschitz continuity of $u$. To this aim, note that, for every fixed $h\in\R^n$, the function $u_h(t,x):=u(t+h,x)$ is a continuous solution to \eqref{appendix eq hj} with initial condition $u_h(0,x)=u(h,x)$. From the fact that $|u(t+h,x)-u(0,x)|\leqslant M(t+h)$ for every $(t,x)\in\ccyl$ one easily see that $u_h$ satisfies condition \eqref{condition} with $\alpha:=2Mh$ and $M_\alpha:=M$. 
By applying Proposition \ref{prop appendix comparison} to the continuous solutions $u_h, u$,  we get 
\[
\|u(t+h,\cdot)-u(t,\cdot)\|_{L^\infty(\R^n)}
\leqslant
\|u(h,\cdot)-u(0,\cdot)\|_{L^\infty(\R^n)}
\leqslant
M\,h.
\]
In other words, the function $u$ is $M$-Lipschitz in $t$. Using the fact that $u$ is a subsolution of \eqref{appendix eq hj}, we get that $u$ satisfies 
\[
 F(x,D_x u)\leqslant -\partial_t u\leqslant M\qquad\hbox{in $\cyl$}
\]
in the viscosity sense. By the coercivity of $F$, we conclude that $u(t,\cdot)$ is Lipschitz for every $t>0$, with a constant only depending on $u_0$ and $F$, see for instance Lemma 2.5 in \cite{barles_book}. 

If $u_0\in\D{UC}(\R^n)$, we can find, for instance by sup-convolution, see \cite{bardi}, a sequence of Lipschitz functions $u^k_0:\R^n\to\R$ that uniformly converge to $u_0$ on $\R^n$. Let us denote by $u^k:\ccyl\to\R$ the  corresponding Lipschitz solution of \eqref{appendix eq hj} with initial datum $u_0^k$. By the comparison principle we have
\[
\|u^m-u^k\|_{L^\infty(\ccyl)}
\leqslant
\|u_0^m-u_0^k\|_{L^\infty(\R^n)},
\]
that is, $(u^k)_k$ is a Cauchy sequence in $\ccyl$ with respect to the sup--norm. Hence the Lipschitz functions $u_k$ uniformly converge to a function $u$ on $\ccyl$, which is therefore uniformly continuous. By the stability of the notion of viscosity solution, we conclude that $u$ is a solution of \eqref{appendix eq hj} with initial datum $u_0$.\qed

\section*{Acknowledgment}
The second author wishes to thank Patrick Bernard for inciting him to write this note and Valentine Roos for explaining him the content and details of her work.

\bibliographystyle{siam}
\bibliography{rigidity}
\bibliographystyle{plain}

\end{document}